\documentclass[12pt,twoside]{article}

\usepackage{amsmath}
\usepackage{amssymb}
\usepackage{amsfonts}
\usepackage{mathrsfs}
\usepackage{graphicx}
\usepackage{epsfig}
\usepackage{accents}
\usepackage{geometry}
\usepackage[pagewise]{lineno}
\usepackage{amscd,amsmath,amssymb,amsthm,amsfonts,epsfig,graphics}
\usepackage{fancyhdr}
\usepackage{tikz}
\usepackage{pgf,pgfplots}
\usepackage{caption}
\usepackage{titlesec}
\usepackage{setspace}

\usepackage{abstract}

\usepackage[numbers, sort]{natbib}
\newcommand \footnoteONLYtext[1]
{
\let \mybackup \thefootnote
\let \thefootnote \relax
\footnotetext{#1}
\let \thefootnote \mybackup
\let \mybackup \imareallyundefinedcommand
}

\usepackage{enumitem}
\setenumerate[1]{itemsep=-4pt, partopsep=0pt, parsep=\parskip, topsep=5pt}
\setitemize[1]{itemsep=-4pt, partopsep=0pt, parsep=\parskip, topsep=5pt}
\setdescription{itemsep=-4pt, partopsep=0pt, parsep=\parskip, topsep=5pt}

\allowdisplaybreaks

\newtheorem{definition}{Definition}[section]
\newtheorem{claim}{Claim}
\theoremstyle{plain}
\newtheorem{thm}{Theorem}
\newtheorem{lem}{Lemma}[section]
\newtheorem{cor}[lem]{Corollary}

\newtheorem{prop}{Proposition}
\newtheorem{remark}{Remark}

\numberwithin{equation}{section}
\numberwithin{figure}{section}
\numberwithin{table}{section}

\newcommand\keywords[1]{\it{Keywords}: #1}
\newcommand\msc[1]{\it{2010 Mathematics Subject Classification}:#1}

\makeatletter
\renewcommand{\section}{\@startsection{section}{1}{0mm}
{-\baselineskip}{0.5\baselineskip}{\normalsize\bf\leftline}}
\makeatother

\makeatletter
\renewcommand{\subsection}{\@startsection{subsection}{1}{0mm}
{-\baselineskip}{0.5\baselineskip}{\normalsize\bf\leftline}}
\makeatother

\title{\Large{\bf{{Lyapunov Exponent and Stochastic Stability for Infinitely Renormalizable Lorenz Maps}}}}

\author{{Haoyang Ji \ \ Qihan Wang}%\footnotemark[1]
\\
}

\date{}

\begin{document}

\maketitle

\vspace{-2cm}

\begin{abstract}
\noindent{ \bf{Abstract.}}
We prove that infinitely renormalizable contracting Lorenz maps with bounded geometry or the so-called {\it a priori bounds} satisfies the slow recurrence condition to the singular point $c$ at its two critical values $c_1^-$ and $c_1^+$. As the first application, we show that the pointwise Lyapunov exponent at $c_1^-$ and $c_1^+$ equals 0. As the second application, we show that such maps are stochastically stable.

\end{abstract}

\footnoteONLYtext{Date: \date{\today}}
\footnoteONLYtext{\msc{\rm{ 37E05} \rm{37H30}}}
\footnoteONLYtext{\keywords{\rm{Lorenz map, renormalization, physical measure, stochastic stability}}}
\section{Introduction}
\label{intro}

 In \cite{Lo} Lorenz studied the solution of the system of differential equations (1.1) in $\mathbb R^3$ originated by truncating  Navier-Stokes equations for modeling atmospheric conditions:
 \begin{align}
\dot{x} &= -10 x + 10y \nonumber \\
\dot{y} & = 28 x - y - xz\\
\dot{z} & = -\frac{8}{3} z + xy \nonumber .
\end{align}
This system exhibits the famous strange Lorenz attractor and has played an important role in the development of the subject of dynamical systems. Guckenheimer and Williams \cite{GW}, and also  Afra$\rm{\breve{i}}$movi$\rm{\check{c}}$-Bykov-Shilnikov \cite{ABS}, introduced the geometric Lorenz flow with the same qualitative behavior as the original Lorenz flow, in which it was supposed that the eigenvalues $\lambda_2< \lambda_1 < 0 < \lambda_3$ at the singularity of the flow satisfying the expanding condition $\lambda_1 + \lambda_3 >0$. In \cite{ACT} Arneodo, Coullet and Tresser began to study a model obtained in the same way just replacing the expanding condition by the contracting condition $\lambda_1 + \lambda_3 <0$. The general assumptions used to construct the geometric models also permit the reduction of the 3-dimension problem, first to a 2-dimensional Poincar\'e section and then to a one-dimensional map, the so-called Lorenz maps.

Hence from a topological viewpoint, a Lorenz map $f : I \setminus \{ c\} \to I$ is nothing else than an interval map with two monotone branches and a discontinuity $c$ in between. On both one sided neighborhoods of the discontinuity the Lorenz map equals $|x|^{\alpha}$ near the origin up to coordinate changes. The parameter $\alpha>0$ is the {\it critical exponent} which by construction equals the ratio of the absolute value between the stable and unstable eigenvalues. If $\alpha < 1$, then the derivative of $f$ at $c$ is infinite. Such maps are typically overall expanding and chaotic, and by this reason these maps are called expanding Lorenz maps. Since $\alpha<1$ holds in the situation of the classical Lorenz systems, expanding Lorenz maps has been studied widely and their dynamics is well understood. If $\alpha>1$, then $f$ is called contracting Lorenz maps. This case is significantly harder due to the interplay between contraction near the discontinuity and expansion outside.

The dynamics of smooth interval maps has been studied exhaustively in the last forty years, especially for unimodal maps. Critical points and critical values play fundamental roles in the study of interval dynamics. From this point of view, Lorenz maps are of hybrid type: these maps have a single critical point as unimodal maps, but two critical values as bimodal maps. The presence of both contraction and discontinuity means that many techniques from the theory of expanding maps and one-dimensional maps are not applicable. However the starting points should still be the refined theory of smooth one-dimensional dynamics, especially of  unimodal maps. The symbolic and topological dynamics of such Lorenz maps have been widely studied, see \cite{B, HS, KP}. The measurable dynamics was studied previously in \cite{CD, KP, Ro, Me1} among others. The first step towards a theory of Lorenz renormalization was taken by Martens and de Melo \cite{MM} who developed a combinatorial counterpart of unimodal renormalization. Further study in this direction can be found in \cite{G, MW, MW1, MW2, W}.

Over the last three decades there has been an increasing interest in stochastic stability. Uniformly expanding maps and uniformly hyperbolic systems are known to be stochastically stable \cite{Ki1}. For non-uniformly expanding interval maps, stochastic stability was previously studied in \cite{BV, BY, Ts} for Benedicks-Carleson-type maps, in \cite{S} under general summability condition, and even in \cite{LW} for unimodal maps with a wild attractor. For contracting Lorenz maps, it was proved in \cite{Me} that Rovella-like maps are stochastically stable in the strong sense of Baladi and Viana \cite{BV}. Such maps exhibit expansion away from a critical region with slow recurrence rate to it and hence admit absolutely continuous invariant measure.

In this paper we are concerned with infinitely renormalizable contracting Lorenz maps which have a global Cantor attractor and therefore have no non-uniformly expanding properties. Such maps possess bounded geometry or the so-called {\it a priori bounds} which guarantees the existence of the unique physical measure supported on the Cantor attractor. {\it A priori bounds} are one of the main ingredients in any study of renormalization and have been established for a large class of Lorenz maps with monotone combinatorics in \cite{MW, G}. We shall consider random perturbations of additive type. Given a map $f$, an $\epsilon$-{\it random (pseudo) orbit} is by definition a sequence $\{ x_n\}_{n=0}^{\infty}$ such that $|f(x_n) - x_{n+1}| \leq \epsilon$. Roughly speaking, stochastic stability means that when $\epsilon>0$ is small, for most of the $\epsilon$-random orbits $\{ x_n\}_{n=0}^{\infty}$, the asymptotic distribution $\lim_{n \to \infty} \frac{1}{n}\sum_{i=0}^{n-1} \delta_{x_i}$ is close to the physical measure of $f$. We will prove that infinitely renormalizable contracting Lorenz maps with bounded geometry are stochastically stable following the strategy of Tsujii \cite{Ts}. We check that such maps satisfy the {\it slow recurrence condition} to the singularity at its two critical values $c_1^-$ and $c_1^+$ and by this the Margulis-Pesin formula holds for the limit measure of random perturbations. As an application of slow recurrence, we also prove that the pointwise Lyapunov exponent equals 0 at $c_1^-$ and $c_1^+$.

This paper is organized as follows. In section 2 we state necessary results and backgrounds which will be used. The precise statements of results will be given in subsection 2.5. In section 3 we check the slow recurrence condition and then prove the pointwise Lyapunov exponent at $c_1^-$ and $c_1^+$ equals 0. The stochastic stability is proved in section 4 under the assumption of Tsujii's theorem.

\section{Preliminaries}

A $C^3$ interval map $f: [0, 1]\setminus \{ c\} \to [0, 1]$ with a discontinuity at $c \in (0, 1)$ is called a {\it Lorenz map} if $f(0) =0, f(1) =1$, $Df(x) >0$ for all $x \in [0, 1]\setminus \{ c\}$. The point $c$ is called the {\it singular point}. A Lorenz map has two critical values defined by $c_1^- = \lim_{x \to c^-}f(x)$ and $c_1^+ = \lim_{x \to c^+}f(x)$, thus implicitly thinking of $c^+$ and $c^-$ as distinct critical points. A Lorenz map is called {\it contracting} provided $Df(c^-) = Df(c^+) =0$. A contracting Lorenz map $f$, with singularity $c$, is called {\it non-flat} if there exists $u \in [0, 1], v \in [0, 1]$, $\alpha > 1$ and $C^3$ diffeomorphisms $\phi : [0, c] \to [0, u^{{1}/{\alpha}}]$ and $\psi : [c, 1] \to [0, v^{{1}/{\alpha}}]$ such that $\phi(c) = 0 =\psi(c)$, $\phi(0) = u^{{1}/{\alpha}}, \psi(1) =v^{{1}/{\alpha}}$ and 
\[ f(x) = \begin{cases}
u - (\phi(x))^{\alpha} & \mbox{ if } x < c\\
1 - v + (\psi(x))^{\alpha} & \mbox{ if } x > c.
\end{cases}\eqno{(1)}
\]
The parameter $\alpha$ is called the {\it critical exponent}, and we also call $c$ the {\it critical point}. Note that $u$ and $1-v$ are the two critical values of $f$.

A Lorenz map is called non-trivial if $c_1^+ < c < c_1^-$. Otherwise all points converge to some fixed point under iteration and for this reason $f$ is called trivial. Unless otherwise noted, all Lorenz maps are assumed to be nontrivial. In general, $c_k^{\pm}$ will denote points in the orbit of the critical values:
\[
c_k^{\pm} = \lim_{x \to c\pm}f^k(c), k \geq 1.
\]

The Schwarzian derivative of a $C^3$ diffeomorphism $h : J \to h(J)$ is denoted by
\[ S h (x) = \frac{D^3 h (x)}{D h (x)} - \frac{3}{2} \left( \frac{D^2 h (x)}{D h (x)}\right)^2 ( Dh (x) \neq 0).
\]
Throughout this article we will always assume that the Lorenz map $f$ has a non-flat critical point $c$ and is of class $C^3$ with negative Schwarzian derivative outside $c$. Furthermore, we may assume that the two fixed points $0$ and $1$ for $f$ are hyperbolic repelling to avoid trivial cases.

\subsection{Renormalization}

Given any interval $J \subset [0, 1]$, the first return map $R_J$ to $J$ for $f$ is defined as $R_J(x) = f^{k(x)}(x)$, for $x \in J \setminus \{c\}$, where $k(x)$ is the smallest positive integer such that $f^{k(x)}(x) \in J$.

\begin{definition}
A Lorenz map $f$ is called {\bf renormalizable} if there exists a closed interval $C$ such that ${\rm int}C \ni c$, $C \neq [0, 1]$, and such that the first return map to $C$ is affinely conjugate to a non-trivial Lorenz map. The interval $C$ is called the renormalization interval. Choose $C$ such that it is maximal with respect to these properties. The rescaled first return map of such $C \setminus \{c\}$ is called a {\bf renormalization} of $f$ and denoted $\mathcal R f$. 
\end{definition}

We will denote
\[ C^-= C \cap [0, c), C^+ = (c, 1],
\]
while the first return map will be denoted $\mathcal P f$ and referred to as the {\it pre-renormalization}. If $f$ is renormalizable, then there exist minimal positive integers $a$ and $b$ such that
\[ \mathcal P f(x) = \begin{cases}
f^{a+1} (x), &x \in C^-,\\
f^{b+1} (x), &x \in C^+.
\end{cases}
\]
Then, explicitly,
\[ \mathcal R f= A^{-1} \circ \mathcal P f \circ A
\]
where $A$ is the affine orientation-preserving rescaling of $[0, 1]$ to $C$. It follows that the left and right boundary points of $C$ are periodic points (of period $a+1$ and $b+1$) and, since $\mathcal R f$ is non-trivial, $C^- \subset f^{a+1}(C^-) \subset C$ and $C^+ \subset f^{b+1}(C^+) \subset C$. Note that $C$ is chosen maximal so that $\mathcal R f$ is uniquely defined. We will explain later why such a maximal interval $C$ exists.

\begin{remark}
We emphasize here that the renormalization is always assumed to be non-trivial. It is possible to define the renormalization operator for maps whose renormalization is trivial but we choose not to include these. Such maps can be thought of as degenerate and including them makes some arguments more difficult which is why they are excluded.
\end{remark}

A branch of $f^n$ is a maximal open interval $J$ on which $f^n$ is monotone (here maximality means that if $A$ is an open interval which properly contains $J$, then $f^n$ is not monotone on $A$). To each branch $J$ of $f^n$ we associate a word $\omega(J) = \sigma_0 \sigma_1 \ldots \sigma_{n-1}$ on symbols $\{ 0, 1 \}$ by 
\[ \sigma_j = \begin{cases}
0 & \mbox{if } f^j(J) \subset (0, c),\\
1 & \mbox{if } f^j(J) \subset (c, 1),
\end{cases}
\]
for $j = 0, 1, \ldots, n-1$.

Next, we wish to describe the combinatorial information encoded in a renormalizable map. The intervals $f^i(C^-), 1 \leq i \leq a$, are pairwise disjoint and disjoint from $C$. So are the intervals $f^i(C^+), 1 \leq i \leq b$. Let $L$ be the branch of $f^{a+1}$ containing $C^-$ and $R$ be the branch of $f^{b+1}$ containing $C^+$. Then we can associate the forward orbits of $C^-$ and $C^+$ to a pair of words $\omega = (\omega^-, \omega^+)$ which will be called the {\it type} of renormalization, where $\omega^- = \omega(L)$ and $\omega^+ = \omega(R)$. In this situation we say that $f$ is $\omega$-renormalizable.

Since $f$ is non-trivial, let $0 < p < c < q < 1$  be the two preimages of $c$. Then $[p, c)$ and $(c, q]$ are maximal intervals adjacent to $c$ such that $f^2$ is monotone on them. Any renormalization interval should be contained in $[p, q]$. In case that $f$ is renormalizable,  let $C'$ be any renormalization interval with renormalization type $\omega'$. By \cite{MM}[Lemma 3.1], the endpoints of $C'$ are hyperbolic repellers since $f$ has negative Schwarzian derivative. Assume $C' = [p', q']$ and $p'$ has period $k$. Then $f^k |[p', c)$ has no other fixed points since for otherwise one can get a contradiction easily by Minimum Principle \cite{MS}[Chpter II, Lemma 6.1]. This shows that for any renormalization type, one can have only one renormalization interval (under the assumption of negative Schwarzian derivative). Now assume that $f$ has two different renormalization type $\omega_1 = (\omega_1^-, \omega_1^+)$ and $\omega_2 = (\omega_2^-, \omega_2^+)$. Let $C_1$ and $C_2$ be the corresponding renormalization interval. By \cite{MM}[Lemma 3.2], we have that $\omega_2^-$ and $\omega_2^+$ are formed by concatenating the words  $\omega_1^-, \omega_1^+$, or vice versa. In particular, 
\[
\omega_2^- = \omega_1^- \omega_1^+ \cdots, \ 
\omega_2^+ = \omega_1^+ \omega_1^- \cdots
\]
and $|\omega_2^-|, |\omega_2^+| \geq |\omega_1^-| + |\omega_1^+|$. Furthermore, $C_1 \subset C_2$ or $C_2 \subset C_1$. So consider the smallest renormalization type, we get a maximal renormalization interval.

Let 
\[ \overline{\omega} = (\omega_0, \omega_1, \ldots) \in \prod_{n \in \mathbb N} \bigotimes(\{0,1 \}^{a_n +1} \times \{0,1 \}^{b_n +1}).
\]
If $\mathcal R^n f$ is $\omega_n$-renormalizable for all $n \in \mathbb N$, then $f$ is called infinitely renormalizable of combinatorial type $\overline{\omega}$. The set of $\omega$-renormalizable maps will be denoted by $\mathcal L_{\omega}$, and the set of infinitely renormalizable maps $f$ such that $\mathcal R^n f$ is $\omega_n$-renormalizable will be denoted by $\mathcal L_{\overline{\omega}}$, $\overline{\omega} = (\omega_0, \omega_1, \ldots) $, with $n$ finite or infinite. If $\overline{\omega}$ is such that $|\omega_n^{\pm}|<B$, $n = 0, 1, \ldots$, for some $0 < B < \infty$, we say that $\overline{\omega}$ is of {\it bounded type}, and $f \in \mathcal L_{\overline{\omega}}$ has {\it bounded combinatorics}.

The combinatorics 
\[ \omega = (0 \overbrace{1 \ldots 1}^{a}, 1 \overbrace{0 \ldots 0}^b)
\]
will be called {\bf monotone}\footnote{Note that the combinatorial description of Lorenz map is simplified due to the fact that Lorenz maps are increasing on each branch so there is no need to introduce permutations as in the case of unimodal maps.} and we also say that $f$ is $(a, b)$-renormalizable. An infinitely renormalizable map is said to be of combinatorial type $\{(a_n, b_n) \}_{n=1}^{\infty}$ if $\mathcal R^{n-1}f$ is $(a_n, b_n)$-renormalizable, for all $n \geq 1$.

\subsection{Covers}

Let $f$ be an infinitely renormalizable map of any combinatorial type. There exists a nested sequence of intervals $C_1 \supset C_2 \supset \ldots \supset C_n \ldots \ni c$ on which the corresponding first return map is again a (non-trivial) Lorenz map. The singular point $c$ splits each $C_n$ into two subintervals denoted by $C_n^- = C_n \cap [0, c)$ and $C_n^+ = C_n \cap (c, 1]$. Let $S_n^-$ and $S_n^+$ denote the first return times of $C_n^-$ and $C_n^+$ to $C_n$, respectively. In particular, the sequences $\{ S_n^-\}_{n \geq 1}$ and $\{ S_n^+\}_{n \geq 1}$ grow at least exponentially fast when $n$ tends to $\infty$:
\begin{equation}
S_n^-, S_n^+ \geq 2^n.
\end{equation}

The $n$-th {\it level cycles} $\Lambda_n^-$ and $\Lambda_n^+$ of $f$ are the following collections of closed intervals
\[ \Lambda_n^- = \bigcup_{k=0}^{S_n^- -1}\overline{f^k(C_n^-)},
\]
and, similarly, for $\Lambda_n^+$. Let $\Lambda_0^- = [0, c]$ and $\Lambda_0^+ = [c, 1]$ and let $\Lambda_n = \Lambda_n^- \cup \Lambda_n^+ $, for all $n \geq 0$. The intervals $\{ \overline{f^k(C_n^-)}\}_{k=0}^{S_n^- -1}$ and $\{ \overline{f^k(C_n^+)}\}_{k=0}^{S_n^+ -1}$ satisfy a disjointness property expressed by the following lemma. Intuitively, for a fixed $n$ these sets have pairwise disjoint interiors except that if they overlap at some time, then all remaining intervals are contained in the same branch and follow the same orbit.

\begin{lem}[\cite{W}, Lemma 2.3.4]
There exists $k_n \geq 0$ such that $\overline{f^{S_n^--i}(C_n^-)}$ and $\overline{f^{S_n^+ - i}(C_n^+)}$ has non-empty intersection alone interior points for $0 \leq i \leq k_n$.
\end{lem}

\begin{proof}
Since the endpoints of $C_n$ are periodic points, each $C_n$ is a nice interval in the sense that $C \ni c$ and the orbit of the boundary of $C_n$ is disjoint from the interior of $C_n$. Since $S_n^-$ and $S_n^+$ are first return times, by Proposition 3.5 in \cite{MW}, there exists a interval $U_n \ni c_1^-$ such that $f^{S_n^- -1} : U_n \to C_n$ is monotone and onto. In particular, $U_n, f(U_n), \cdots, f^{S_n^- -1}(U_n)$ have pairwise disjoint interiors.  Similarly, there exists $V_n \ni c_1^+$ such that $f^{S_n^+ -1} : V_n \to C_n$ is monotone and onto.

Let $k_n \geq 0$ be the largest integer such that  $f^{S_n^- -1 -i}(U_n) = f^{S_n^+ -1 -i}(V_n)$ for all $0 \leq i \leq k_n$. Such an integer $k_n$ exists since $f^{S_n^- -1}(U_n) = C_n = f^{S_n^+ -1}(V_n)$. Then for each $0 \leq i \leq k_n$, $\overline{f^{S_n^--i}(C_n^-)}$ and $\overline{f^{S_n^+ - i}(C_n^+)}$ are both contained in $f^{S_n^- -1 -i}(U_n)$. Since $f^i : f^{S_n^- -1 -i}(U_n) \to C_n$ is monotone and onto, and each renormalization is assumed to be non-trivial, the interiors of $\overline{f^{S_n^--i}(C_n^-)}$ and $\overline{f^{S_n^+ - i}(C_n^+)}$ must have non-empty intersection.
\end{proof}

\begin{remark}
If $f$ has monotone combinatorics, then $k_n = 0$. If $k_n \geq 1$, then 
\[
\Lambda_n = \bigcup_{k=0}^{S_n^- -k_n-1}\overline{f^k(C_n^-)} \cup \bigcup_{k=0}^{S_n^+ -k_n-1}\overline{f^k(C_n^+)} \cup \bigcup_{k=S_n^- -k_n}^{S_n^- -1} f^k(U_n)
\]
such that the interiors of elements in $\{ \Lambda_n \}$ are pairwise disjoint.
\end{remark}

Components of $\Lambda_n$ are called intervals of generation $n$ and components of $\Lambda_{n-1}\setminus \Lambda_n$ are called gaps of generation $n$.  Let $J \subset I$ be intervals of generation $n+1$ and $n$, respectively, and let $G \subset I$ be a gap of generation $n+1$. 
The intersection of all levels is denoted by
\[
\Lambda  = \bigcap_{n \geq 0} \Lambda_n.
\]

\begin{definition}
An infinitely renormalizable contracting Lorenz map $f$ is called having {\bf bounded geometry} if there exist constants $K>1$ and $0 < \mu < \lambda < 1$ independent of $n$ such that for all $n \geq 1$, 
\[ 
\mu< \frac{|J|}{|I|},  \frac{|G|}{|I|} < \lambda 
\]
and
\[
\frac{|C_n|}{|C_n^{\pm}|} \leq K.
\]
\end{definition}

\begin{definition}
Let $\mathcal L_B$ denote the set of infinitely renormalizable contracting Lorenz maps with bounded geometry.
\end{definition}

\begin{lem}
Suppose $f \in \mathcal L_B$, then $f$ has bounded combinatorics.
\end{lem}

\begin{proof}
The number of intervals of generation $n+1$ which are contained in $C_n = C_n^- \cup C_n^+$ are exactly $\omega_n^- + \omega_n^+-1$. Assume that $C_n^-$ contains at least $(\omega_n^- + \omega_n^+-1)/2$ intervals of generation $n+1$, then $\omega_n^- + \omega_n^+ \leq 2/\mu +1$. This finishes the proof.

\end{proof}

In interval dynamics, the property bounded geometry (or called the {\it a priori bounds}) was first proved by Sullivan for infinitely renormalizable unimodal maps with bounded combinatorics. Bounded geometry is also the fundamental result in the study of critical circle maps. For background and history, see \cite{MS} and the references therein.

The class $\mathcal L_B$ is non-empty. The first result was given by Martens and Winckler in \cite{MW}, where the {\it a priori bounds} was proved for special monotone combinatorics with the return time of one branch being large and much larger than the return time of the other branch. Roughly speaking, \cite{MW} proved the {\it a priori bounds} for monotone combinatorial types with the following return time:
\begin{equation}
 [\alpha] \leq |\omega^-| -1 \leq [2 \alpha -1], \ b_- \leq |\omega^+| -1 \leq b_+,
\end{equation}
where $b_-$ is sufficiently large and $b_+$ depends on the choice of $b_-$. After that, Gaidashev in \cite{G} proved the {\it a priori bounds} for a different class of monotone combinatorial types with sufficiently flat critical point. The range of the allowed length of the combinatorics is close to 
\[ 
\left(\frac{\ln 2 }{ \ln \alpha} +1\right) \alpha < |\omega^-|, |\omega^+| < 2 \alpha.
\]
In comparison to (2.2), the length of combinatorics is similar to the one for the shorter branch in \cite{MW}.

\begin{remark}
We emphasize here that our results (including Theorem 1,2 \& 3) hold under the assumption of {\bf bounded geometry} or the so-called {\it a priori bounds}, which implies bounded combinatorics by Lemma 2.1. Up to now, the bounded geometry property has been verified for a large class of Lorenz maps with monotone combinatorics.  {\color{red} It is proved recently by Martens and Winckler \cite{MW2} that there exist contracting Lorenz maps with bounded combinatorics and no a priori bounds.}
\end{remark}

\subsection{Physical measures}

Let $I$ be an interval, let $g : I \to I$ and let $\delta_x$ denote the Dirac measure at $x$. An $g$-invariant measure $\mu : I \to \mathbb R$ is called a {\it physical measure} if its {\it basin}
\[
B(\mu)= \{ x \in I :  \frac{1}{n} \sum_{k=1}^n \delta_{g^k(x)} \to \mu \mbox{ as } n \to \infty \mbox{ in the weak star topology} \}
\]
has positive Lebesgue measure.

Suppose $f \in \mathcal L_B$,  let $\mathcal O_f$ be the closure of the orbits of the critical values. The following proposition was proved in \cite{MW}, see also \cite{B, G, W}.

\begin{prop}
Assume $f \in \mathcal L_B$. Then:
\begin{enumerate}
\item[(1)] $\mathcal O_f = \Lambda$ is a minimal Cantor set;
\item[(2)] $\mathcal O_f$ has zero Lebesgue measure;
\item[(3)] $f| \mathcal O_f$ is uniquely ergodic;
\item[(4)] $\mathcal O_f$ is the global attractor of $f$ whose basin of attraction has full Lebesgue measure.
\end{enumerate}
\end{prop}

\begin{proof}
The proof is identical to \cite{MW}[Theorem 5.5]. Firstly, By Singer's theorem\footnote{Singer's theorem is stated for unimodal maps but the statement and proof can easily be adapted to contracting Lorenz maps.} \cite{MS}[Theorem 2.7], the immediate basin of any periodic attractor contains at least one of the critical values (recall that we assume the endpoints of $I$ to be repelling). Since $f$ is infinitely renormalizable, the critical orbits have subsequences which converge on the singular point and the critical values are not periodic,  $f$ has no periodic attractors.

Next we show that $\mathcal O_f = \Lambda$. Clearly $\mathcal O_f \subset \Lambda$ since the critical values are contained in $\overline{f(C_n^-)} \cup \overline{ f(C_n^+)}$ for each $n$. By bounded geometry assumption, $|\Lambda_{n+1}| < \lambda |\Lambda_n|$ so the lengths of the intervals of generation $n$ tend to 0 as $n \to \infty$. Hence $\mathcal O_f = \Lambda$.

A standard argument demonstrates that $\mathcal O_f$ is a Cantor set of measure 0 (since $\lambda <1$) and of Hausdorff dimension in $(0, 1)$.

By Lemma 2.2, $f$ has bounded geometry. So the unique ergodicity follows from a theorem due to Gambaudo and Martens \cite{GM}. The original proof in \cite{GM} was stated for continuous maps. For an adaption to contracting Lorenz maps, see \cite{W}[Theorem 2.3.1].

Finally, by Proposition 3.7 and Theorem 3.10 in \cite{MW}, $f$ has no wandering intervals and almost all points are attracted to $\mathcal O_f$.

\end{proof}

\begin{remark}
A {\bf wandering interval} of $f$ is an open interval $J$ such that $f^i(J)$ are pairwise disjoint, for all $i \geq 1$, and the $\omega$-limit set of $J$ is not a single periodic orbit. The non-existence of wandering intervals for $C^3$ contracting Lorenz maps with negative Schwarzian derivative was established by Cui and Ding \cite{CD} under non-uniformly hyperbolic condition and by Martens and Winckler \cite{MW} under weak Markov property. It is conjectured by Martens and de Melo in \cite{MM} that if $f$ has wandering intervals then $f$ has Cherry attractor.
\end{remark}

By Proposition 1, $f$ has a unique physical measure $\mu$ supported on $\mathcal O_f$. Since
\[
\mu\left( \cup_{k=0}^{S_n^- -1} (f^k(C_n^- )\cap \mathcal O_f) \right) \leq 1
\]
and 
\[
\mu (f^k(C_n^-) \cap \mathcal O_f) \geq \mu(C_n^- \cap \mathcal O_f) \mbox{ for all } 1 \leq k \leq S_n^- -1,
\]
we have 
\[
\mu(C_n^- \cap \mathcal O_f) \leq 1/S_n^- \mbox{ and similarly, }\mu(C_n^+ \cap \mathcal O_f) \leq 1/S_n^+.
\]
Furthermore, 
\begin{equation}
\mu(C_n \cap \mathcal O_f) \leq 2/S_n \mbox{ where } S_n = \min\{ S_n^-, S_n^+ \}.
\end{equation}

%%%%%%%%%%%%%%%%%%%%%%%%%%%%%%%%%%%

\subsection{Statement of results}

The {\it pre-orbit} of a point $x \in [0, 1]$ is the set $O_f^-(x) : = \bigcup_{n \geq 0} f^{-n}(x)$. For a point $x \in [0, 1] \setminus O_f^-(c)$, denote the forward orbit of $x$ by $O_f^+(x) = \{ f^j(x); j \geq 0\}$. And, if there exists $p \geq 1$ such that $f^p(c^-) =c $, where $p$ is the smallest positive integer with this property, we define $O_f^+(c^-) = \{f^j(c^-); 1 \leq j \leq p \}$. Otherwise we define $O_f^+(c^-) = \{ f^j(c^-): j \geq 0\}$. Similarly we define $O_f^+(c^+)$.

Given any point $x \in [0,1] \setminus O_f^-(c) $, we say that $x$ satisfies the {\it slow recurrence} condition to $c$ provided
\[ \lim_{\delta \to 0} \liminf_{n \to \infty} \frac{1}{n} \sum_{\substack{ 0 \leq i < n \\ f^i(x) \in \mathcal C_{\delta} }} \log d(f^i(x), c) = 0.
\]
Here $d(x, c) : =|x - c|  $ and $\mathcal C_{\delta} : =  (c - \delta, c) \cup (c,  c+ \delta)$ is a punctured neighborhood of the critical point.

\begin{thm}
Suppose $f \in \mathcal L_B$, then its two critical values $c_1^-$ and $c_1^+$ satisfy the slow recurrence condition to $c$. 
\end{thm}

The Lyapunov exponent at $x \in [0,1] \setminus O_f^-(c)$, denoted $\chi_f(x)$, is given by 
\[ \chi_f(x) = \lim_{n \to \infty} \frac{1}{n} \log Df^n(x),
\]
provided the limit exists. Otherwise one can consider the upper Lyapunov exponent
\[\overline\chi_f(x) = \limsup_{n \to \infty} \frac{1}{n} \log Df^n(x).
\]
For an $f$-invariant Borel probability measure $\nu$, its Lyapunov exponent, denoted $\chi_{\nu}(f)$, is given by \[
\chi_{\nu}(f) = \int \log Df d \nu.\]

\begin{thm}
Suppose $f \in \mathcal L_B$, let $\mu$ be the unique physical measure of $f$. Then $\chi_f(c_1^-)= \chi_f(c_1^+)=0$. Moreover, $ \chi_{\mu} (f) =  0$.
\end{thm}

Now we consider random perturbations of infinitely renormalizable contracting Lorenz maps.

Denote $I = [0,1]$. From now on, we shall assume that $f(I) \subset {\rm int} (I)$ and let $\epsilon_0 = d(f[0, 1], \{ 0, 1\})$. This is reasonable since the endpoints of $I$ are assumed to be repelling, so we can extend $f$ to a little larger interval than $I$ and then rescale to $I$ affinely. In this sense, the endpoints are no more fixed. Then $f_t(I) \subset I$ for all $|t| \leq \epsilon_0$, where $f_t(x) = f(x) +t$. Let $T= [-\epsilon_0, \epsilon_0]$, and for $\underline t =(t_1, t_2, \ldots) \in T^{\mathbb N}$, define
\[ f_{\underline t}^n = f_{t_n} \circ f_{t_{n-1}} \circ \ldots \circ f_{t_1}, \ \ \ n =1, 2 \ldots.
\]
We call $\{ f_{\underline t}^n(x)\}_{n=0}^{\infty}$ a {\it random orbit} starting from $x$. For $\epsilon \in (0, \epsilon_0]$, let $\theta_{\epsilon}$ be a probability measure supported on $[-\epsilon, \epsilon]$. This naturally induces a Markov process $\chi_{\epsilon}$ on $I$ with the following transition probabilities:
\[ P_{\epsilon}(E|x) = \theta_{\epsilon} \{ t | t \in [-\epsilon, \epsilon] , f_t(x) \in E \} = \int_E \theta_{\epsilon} (y - f(x)) dy.
\]
Then each $P_{\epsilon}$ is supported in $ [-\epsilon, \epsilon]$. In what follows, we consider the following conditions on the probability measure $\theta_{\epsilon}$:
\begin{itemize}
\item[(A1)] each $\theta_{\epsilon}$ is supported on $[-\epsilon, \epsilon]$ and is absolutely continuous with respect to the Lebesgue measure;
\item[(A2)] there exists $d_0>0$ such that for all $\epsilon >0$, the density function satisfies $|\frac{{\rm d\theta_{\epsilon}}}{{\rm d Leb}}|_{L^{\infty}} < \frac{d_0}{\epsilon}$ and $\frac{{\rm d}\theta_{\epsilon}}{{\rm d Leb}} >0$ in a neighborhood of 0.
\end{itemize}

A measure $\mu_{\epsilon}$ on $I$ is called a {\it stationary measure} for the Markov process $\chi_{\epsilon}$ (or for $\theta_{\epsilon}$), if for any Borel set $A$ on $I$, we have
\[ \mu_{\epsilon} (A) = \int_I P_{\epsilon}(A|x) d\mu_{\epsilon} (x) = \int_{-\epsilon}^{\epsilon} \mu_{\epsilon}(f_t^{-1}(A)) d \theta_{\epsilon}(t).
\]
In other words, for any continuous map $\varphi: I \to \mathbb R$, the following holds
\[ \int \varphi(x) d\mu_{\epsilon} (x) = \int \int \varphi(f_t(x)) d \mu_{\epsilon} (x) d \theta_{\epsilon} (t).
\]
It follows from condition $({\rm A1})$ and $({\rm A2})$ that, for any $\epsilon>0$ small enough, the Markov process $\chi_{\epsilon}$ has a unique stationary measure $\mu_{\epsilon}$. The existence follows by for example \cite{AA}[Lemma 3.5]. The uniqueness comes from the property that the density function ${\rm d}\theta_{\epsilon}/{\rm d Leb}$ is bounded from below. For a proof, see \cite{BY}[Part II]. The uniqueness also implies that $\mu_{\epsilon}$ is ergodic, \cite{Ki1}[Theorem 2.1]. Moreover, $\mu_{\epsilon}$ is absolutely continuous with respect to Lebesgue, see for example \cite{S}[Lemma 3.2].

A classical result in random dynamical systems (see for example \cite{AA}[Remark 3.1]) implies that every weak star accumulation point of the stationary measures $\mu_{\epsilon}$ when $\epsilon \to 0$ is an $f$-invariant probability measure, which is called a {\it zero-noise limit measure}. This naturally leads to the study of the kind of zero noise limits which arise to the notion of stochastic stability.

\begin{definition}[{\bf Stochastic stability}]
Let $f$ be a contracting Lorenz map with $f(I) \subset {\rm int}(I)$ and let $\epsilon_0 = d(f[0, 1], \{ 0, 1\})$. Suppose that $f$ has a unique physical measure $\mu$. We say that $f$ is stochastically stable with respect to the family $\{\theta_{\epsilon} \}_{0 < \epsilon \leq \epsilon_0}$ if $\mu_{\epsilon}$ converges to the physical measure $\mu$ in the weak star topology as $\epsilon \to 0$.
\end{definition}

\begin{thm}
Suppose $f \in \mathcal L_{B}$, then $f$ is stochastic stable with respect to the family $\{\theta_{\epsilon} \}_{0 < \epsilon \leq \epsilon_0}$ under the assumption $({\rm A1})$ and $({\rm A2})$.
\end{thm}

\begin{remark}
Theorem 3 can be strengthened as follows. Let $f$ be a $C^3$ contracting Lorenz map with negative Schwarzian derivative. If $f$ satisfies
\begin{itemize}
\item[(1)] $f$ has a Cantor attractor $A$ and no wandering intervals; 
\item[(2)] $c_1^{\pm}$ satisfies the slow recurrence condition to $c$;
\item[(3)] $f$ has a unique physical measure $\mu$ supported on $A$ such that $f|A$ is uniquely ergodic.
\end{itemize}
Then $f$ is stochastically stable. For example, if $f$ has a wild attractor but has no wandering intervals, then $f$ may have a physical measure. If $f$ further satisfies condition (2) and (3), then $f$ is stochastically stable. However, many questions are still open about wild attractor for contracting Lorenz maps.
\end{remark}

\section{Lyapunov exponent}

In this section we check the slow recurrence condition for Lorenz maps from class $\mathcal L_B$. Using this result we prove the integrability of $\log Df$ and show that the pointwise Lyapunov exponent at $c_1^{\pm}$ equals 0.

The following lemma clarifies the geometric behavior of a map near a non-flat critical point and will be useful.

\begin{lem}
Given a contracting Lorenz map $f$ with non-flat critical point of critical exponent $\alpha>1$ and negative Schwarzian derivative, there exists a neighborhood $U$ of $c$ such that:
\begin{itemize}
\item[(1)] There exist constants $0 < a < b$ such that for all $x \in U \setminus \{c\}$
\[ a |x - c|^{\alpha-1} < Df(x) < b |x - c|^{\alpha-1}.
\]
\item[(2)] There exists $C_0>0$ such that for all $x \in U \setminus \{c\}$
\[ \log Df(x) \geq C_0 \log |x-c|.
\]
\end{itemize}
\end{lem}

\begin{proof}
If $x <c$, it follows from the definition of non-flatness and Taylor's formula, that
\[ \lim_{x \to c^-} \left( \frac{Df(x)}{|x-c|^{\alpha-1}}\right) =  \lim_{x \to c^-} \alpha D\phi(x) \left(  \frac{\phi(x)}{|x-c|} \right)^{\alpha-1} = \alpha (D \phi(c))^{\alpha} >0.
\]
Statement (2) follows from statement (1). 
\end{proof}

\subsection{Slow recurrence}

To prove Theorem 1, we need the following lemma.

\begin{lem}

Let $x = c_1^{\pm}$, then for any $k \geq 1$ and $n$, we have
\[ \# \{ 0 \leq i < n : f^i(x) \in C_k \} \leq \frac{n+1}{S_k},
\]
where $S_k = \min\{ S_k^-, S_k^+\}$.
\end{lem}

\begin{proof}

It suffices to prove for $x = c_1^+$. For any $k \geq 1$, the first entry time of $c_1^+$ to $C_k$ is $S_k^{+} -1 \geq S_k -1$ from the definition of renormalization. In particular, if $n+1 < S_k$, then $ \# \{ 0 \leq i < n : f^i(x) \in C_k \}$ is actually 0 for $x = c_1^{\pm}$. Also the return time of any $y \in \mathcal O_f \cap C_k$ to $C_k$ is at least $S_k = \min\{ S_k^-, S_k^+\}$. Therefore,
\[ \# \{ 0 \leq i < n : f^i(x) \in C_k \} \leq \lfloor \frac{n+1}{S_k}\rfloor \leq \frac{n+1}{S_k},
\]
where $\lfloor \cdot \rfloor$ is the integer floor function. 
\end{proof}

Note that Birkhoff's Ergodic Theorem and the fact $\mu(C_k) \leq 2/S_k$ implies that $ \# \{ 0 \leq i < n : f^i(x) \in C_k \} \sim n \mu(C_k) \leq 2n/S_k$ for $n$ large enough and for $\mu$-typical $x \in \mathcal O_f$. However, we can not use unique ergodicity to prove that this holds for $x = c_1^{\pm}$ since the characteristic function $1_{C_k}$ is not continuous.

\begin{proof}[{\bf Proof of Theorem 1}]

As explained above, let $\mu$ be the unique invariant Borel probability measure of $f$ supported on $\mathcal O_f$. By bounded geometry, there exist $0 < \rho < \rho' < 1$ such that $\rho < |C_{k+1}^-|/|C_k^-|, |C_{k+1}^+|/|C_k^+|  < \rho'$ for all $k \geq 1$. Then there exists $C_1 > 0$ such that $|C_k^-|, |C_k^+| \geq C_1 \cdot \rho^k$, and also $|C_k| \geq 2C_1 \cdot \rho^k$. Now for any $\delta>0$ sufficiently small, there exists $k_0 > 0$ maximal such that $\mathcal C_{\delta} \subset  C_{k_0}$. Moreover, $k_0 \to \infty$ as $\delta \to 0$.

Recall that $S_k$ grows at least exponentially fast: $S_k \geq 2^{k}$. By Lemma 3.2, for $x = c_1^{\pm}$, we have
\begin{align*}
\frac{1}{n}  \sum_{\substack{0 \leq i <n \\ f^i(x) \in \mathcal C_{\delta}}} -\log d (f^i(x), c) 
& \leq \frac{1}{n} \sum_{k \geq k_0} \sum_{\substack{0 \leq i <n  \\ f^i(x) \in C_k \setminus C_{k+1} }} - \log \min\{|C_{k+1}^-|, |C_{k+1}^+|\}\\
& \leq \frac{1}{n} \sum_{k \geq k_0} - \log \min\{|C_{k+1}^-|, |C_{k+1}^+|\} \cdot \frac{n+1}{S_{k}}\\
& = \left(1 + \frac{1}{n} \right) \sum_{k \geq k_0} \frac{1}{S_{k}} \cdot ( -\log \min\{|C_{k+1}^-|, |C_{k+1}^+|\})\\
& \leq 2\sum_{k \geq k_0} \frac{1}{2^{k}} \cdot ( - (k+1)\log \rho - \log C_1)\\
& \leq - \log \rho \sum_{k \geq k_0} \frac{k+1}{2^{k-1}} - \log C_1 \sum_{k \geq k_0} \frac{1}{2^{k-1}}.
\end{align*}
The last term tends to $0$ as $k_0 \to \infty$ and hence as $\delta \to 0$. This shows that for any $\epsilon> 0$, there exists $\delta >0$ small enough, and for any $n>1$, we have 
\[  \frac{1}{n} \sum_{\substack{0 \leq i <n \\ f^i(x) \in \mathcal C_{\delta}}} -\log d (f^i(x), c) < \epsilon.
\]
Hence 
\[ \limsup_{n \to \infty} \frac{1}{n} \sum_{\substack{0 \leq i <n \\ f^i(x) \in \mathcal C_{\delta}}} -\log d (f^i(x), c) \leq \epsilon,
\]
here $x = c_1^{\pm}$. This finishes the proof.

\end{proof}

\subsection{Integrability}

The integrability of $\log Df$ for smooth interval maps was obtained by Przytycki in \cite{Pr}, where he also proved that the Lyapunov exponent is non-negative: $\int \log Df d \mu \geq 0$. We will obtain the integrability again using bounded geometry. The proof is similar with \cite{dFG2}[Proposition 3.1].

Let $\psi(x) = | \log Df(x)|$. Let $\psi_n =1_{[0,1]\setminus C_n} \cdot \psi$, that is, $\psi_n =0$ on $C_n$ and $\psi_n = \varphi$ outside $C_n$.

\begin{prop}
Suppose $f \in \mathcal L_B$, then the function $\log Df$ is $\mu$-integrable, i.e., $\int |\log Df| d\mu < \infty$.
\end{prop}

\begin{proof}

Note that the sequence $\{\psi_n \}$ converges monotonically to $\psi(x)$. Let $n_0$ be the smallest positive integer such that $C_{n_0} \subset U$ as in Lemma 3.1. It suffices to consider the values for $n \geq n_0$. Since $\psi_n$ is identically zero on $C_n$ and equals $\psi$ everywhere else, we have

\[ \int \psi_n d \mu \leq \int_{[0, 1]\setminus C_{n_0}}  \psi d \mu + \sum_{k=n_0}^{n-1} \int_{C_k \setminus C_{k+1}}  \psi d \mu. \eqno{(3.1)}
\]
The first integral on the right-hand side is a fixed number independent of $n$. Hence it suffices to bound the last sum. Using Lemma 3.1, we have
\[ \sum_{k=n_0}^{n-1} \int_{C_k \setminus C_{k+1}}  \psi d \mu \leq - C_0 \sum_{k=n_0}^{n-1} \mu(C_k \setminus C_{k+1}) \log \min\{|C_{k+1}^-|, |C_{k+1}^+|\}. \eqno{(3.2)}
\]
By (2.1) and (2.3), we have that
\[
\mu(C_k \setminus C_{k+1}) \leq \mu (C_k \cap \mathcal O_f) \leq \frac{2}{S_k} \leq \frac{1}{2^{k-1}}.
\]
 Therefore, 
\begin{align*}
& - C_0 \sum_{k=n_0}^{n-1} \mu(C_k \setminus C_{k+1}) \log \min\{|C_{k+1}^-|, |C_{k+1}^+|\} \\ & \leq  - C_0 \sum_{k=n_0}^{n-1} \frac{1}{2^{k-1}} (\log C_1 + (k+1) \log \rho)\\
& =  -C_0 \log C_1 \sum_{k=n_0}^{n-1} \frac{1}{2^{k-1}} + (-C_0 \log \rho) \sum_{k=n_0}^{n-1} \frac{k+1}{2^{k-1}} < \infty.
\end{align*} 
Taking back to (3.1), we can conclude that there exists a constant $C_2>0$ independent of $n$ such that 
\[ \int \psi_n d \mu \leq C_2 \mbox{ for all } n \geq 1.
\]
Then by the Monotone Convergence Theorem, $\psi$ is $\mu$-integrable. This finishes the proof.
\end{proof}

\subsection{Zero Lyapunov exponent}

In this subsection we use a different truncated map which is continuous and used in \cite{Ji}. First we remark that since $f|\mathcal O_f$ is uniquely ergodic, for any continuous map $\phi : \mathcal O_f \to \mathbb R$ and any $x \in \mathcal O_f$, we have
\[ \lim_{n \to \infty} \frac{1}{n} \sum_{i=0}^{n-1} \phi (f^i(x)) = \int \phi d \mu.
\]
Let $\varphi(x) = \log Df(x)$, and $\varphi_N(x) = \max(\varphi(x), -N)$, here $N = 1, 2, 3, \ldots$. Obviously, $\varphi_N(x) \geq \varphi(x)$. By Lemma 3.1, for any $N$ large enough, there exists $\delta_N>0$ small enough such that $\varphi(x) \geq -N$ outside $\mathcal C_{\delta_N}$. Moreover, $\delta_N \to 0$ as $N \to \infty$.

\begin{lem}
Suppose $f \in \mathcal L_B$. Then $\chi_f(c_1^-) = \chi_f(c_1^+) = \chi_{\mu} (f) = \int \log Df d\mu$.
\end{lem}

\begin{proof}
Let $x = c_1^{\pm}$, we have
\begin{align*}
0 & \leq  \frac{1}{n} \sum_{i=0}^{n-1} \varphi_N (f^i(x)) - \frac{1}{n} \sum_{i=0}^{n-1} \varphi (f^i(x)) \\
   & = \frac{1}{n} \sum_{\substack{0 \leq i < n \\ \varphi(f^i(x)) < -N}} (-N -\varphi(f^i(x)))\\
   & = \frac{1}{n} \sum_{\substack{0 \leq i < n \\ \varphi(f^i(x)) <- N}} ( -N - \log Df(f^i(x)) ).
\end{align*}
By Lemma 3.1 and the remark before the beginning of the proof, we have that for $N$ large enough
\begin{equation*}
\frac{1}{n} \sum_{\substack{0 \leq i < n \\ \varphi(f^i(x)) < - N}} (-N - \log Df(f^i(x)) )  \leq C_0 \cdot \frac{1}{n} \sum_{\substack{0 \leq i < n \\ f^i(x)\in C_{\delta_N}}} - \log |f^i(x) - c|.
\end{equation*}
By Theorem 1, the last term is arbitrarily small provided $n$ and $N$ are large enough. On the other hand, since $\varphi_N(x)$ is continuous, the sequence of time averages
\[ \lim_{n \to \infty} \frac{1}{n} \sum_{i=0}^{n-1} \varphi_N(f^i(x)) 
\]
converges at every $x \in \mathcal O_f$ to $\int \varphi_N d \mu$. By Proposition 2 and the monotone convergence theorem,
\[\lim_{N \to \infty} \int \varphi_N d \mu = \int \varphi d \mu = \chi_{\mu} (f).
\]
It follows that $n^{-1} \sum_{i=0}^{n-1} \varphi(f^i(x)) $ is also a Cauchy sequence, hence its limit exists when $n \to \infty$. Furthermore, 
\[ \lim_{n \to \infty } \frac{1}{n} \sum_{i=0}^{n-1} \varphi(f^i(x))  = \lim_{N \to \infty }\lim_{n \to \infty} \frac{1}{n} \sum_{i=0}^{n-1} \varphi_N(f^i(x))  = \int \varphi d \mu = \chi_{\mu} (f).
\]
This finishes the proof. 
\end{proof}

According to Przytycki (\cite{Pr}[Theorem B]), if $f$ is an interval maps and $\mu$ is an ergodic invariant probability, then either $\chi_{\mu}(f) \geq 0$ or $\mu$ is supported on a strictly attracting periodic orbit. Since $f$ is infinitely renormalizable and $Sf \leq 0$, $f$ has no attracting periodic cycles.

\begin{proof}[{\bf Proof of Theorem 2}]
From the remark above, $\chi_{\mu}(f) \geq 0$. Assume that $\chi_{\mu}(f) >0$. Then by Lemma 3.3, $\chi_f(c_1^+) = \chi_f(c_1^-) >0$ which implies the large derivative condition:
\[ \lim_{n \to \infty} Df^n(c_1^{\pm}) = + \infty.
\]
According to \cite{CD} there exists an absolutely continuous invariant probability measure. This contradicts the fact that $f$ has a physical measure supported on Cantor attractor $\mathcal O_f$. So $\chi_{\mu}(f) = 0$ and the proof is finished.
\end{proof}

\section{Stochastic stability}

\subsection{Margulis–Pesin entropy formula}

In \cite{Ts}, Tsujii considered the random perturbations of multimodal interval maps with non-degenerate critical points. By modifying Tsujii's proof, we have the following theorem whose proof will be stated in subsection 4.2.

\begin{thm}
Let $f$ be a contracting Lorenz map with non-flat critical point and such that
$f(I) \subset {\rm int}(I)$. Let $\epsilon_0 = d(f[0, 1], \{ 0, 1\})$.  Suppose that $c_1^{\pm}$ satisfy the slow recurrence condition to $c$. For any $0 < \epsilon < \epsilon_0$, let $\theta_{\epsilon}$ be a Borel probability measure satisfying condition $({\rm A1})$ and $({\rm A2})$. If a sequence of stationary measures $\mu_{\epsilon}$ for $\theta_{\epsilon}$ converges to a measure $\mu_{\infty}$, then $\mu_{\infty}$ satisfies the Margulis-Pesin entropy formula.
\end{thm}

Here the Margulis-Pesin entropy formula is the following formula for $f$-invariant probability measure $\mu$:
\[ h_{\mu}(f) = \int \chi^+(x) {\rm d} \mu(x)
\]
where the left side is the metric entropy and $\chi^+(x) := \max \{\overline\chi_f(x), 0 \}$.

Let $\sigma: T^{\mathbb N} \to T^{\mathbb N}$ be the shift map and define the skew product $F: I \times T^{\mathbb N} \to  I \times T^{\mathbb N}$ as follows: for any $\underline t =(t_1, t_2, \ldots) \in T^{\mathbb N}$
\[ \sigma(\underline t) = (t_2, t_3, \ldots), \ \ F(x, \underline t) = (f_{t_1}(x),  \sigma(\underline t)).
\]
A stationary measure $\mu_{\epsilon}$ is called {\it ergodic} if $\mu_{\epsilon} \times \theta_{\epsilon}^{\mathbb N}$ is ergodic for skew product map $F$. The existence of stationary measure is well known. The uniqueness of $\mu_{\epsilon}$ comes from condition (A2) assuming that the density of $\theta_{\epsilon}$ is bounded from below. The ergodicity follows from the uniqueness, see for example \cite{Ki1}[Theorem 2.1] and \cite{LW}[Lemma 4.1].

Let $\pi :  I \times T^{\mathbb N} \to I$ be the projection. Let $\mathscr A$ be the sub-$\sigma$-algebra of Borel $\sigma$-algebra of $I \times T^{\mathbb N}$ which consists of all the subsets of the form $I \times B$ with $B$ a Borel subset of $T^{\mathbb N}$. We also define the entropy $h_{\mu}(\theta)$ for a probability measure $\theta$ on $T$ and $P_{\theta}$-invariant probability measure $\mu$ on $I$ by
\[ h_{\mu}(\theta) = h_{\mu \times \theta^{\mathbb N}}(F| \mathscr A) : = \sup_{\mathcal Q} h_{\mu \times \theta^{\mathbb N}}(F, \mathcal Q|\mathscr A)
\]
where $\mathcal Q$ is finite partition of $I \times T^{\mathbb N}$, the right side is ordinary condition entropy. If $\mu$ is an $f$-invariant measure, then $h_{\mu}(\delta_f) = h_{\mu}(f)$. For $(x, \underline t) \in I \times T^{\mathbb N}$, the Lyapunov exponent for the random trajectories, denoted $\chi(x, \underline t)$, is given by
\[ \chi(x, \underline t) =\lim_{n \to \infty} \frac{1}{n} \log Df_{\underline t}^n (x).
\]
By Birkhoff again, if $\mu \times \theta^{\mathbb N}$ is ergodic, then $\chi(x, \underline t)$ exists and does not depend on $\underline t$ for $\mu \times \theta^{\mathbb N}$ almost every $(x, \underline t)$. For this reason we shall denote $\chi(x, \underline t)$ as $\chi(x; \theta)$, see \cite{Ki1}[Theorem 2.2].

%need a proof?

\begin{prop}
Let $f$ be a contracting Lorenz map with non-flat critical point. Then there exists a constant $\delta > 0$ such that if $\epsilon < \delta$ and $\theta_{\epsilon}$ is a probability measure on $[- \epsilon, \epsilon]$ which is absolutely continuous w.r.t the Lebesgue measure with $|\frac{{\rm d\theta_{\epsilon}}}{{\rm d Leb}}|_{L^{\infty}} < \infty$. Then 
\[ h_{\mu_{\epsilon}}(\theta_{\epsilon}) = \int \chi^+(x; \theta_{\epsilon}) d \mu_{\epsilon}(x).
\]
Where $\chi^+(x; \theta_{\epsilon}) = \max \{ 0, \chi(x; \theta_{\epsilon})\}$ and $\mu_{\epsilon}$ is the stationary measure for $\theta_{\epsilon}$.
\end{prop}

This formula was proved for smooth interval map in \cite{Ts}, and for random diffeomorphisms in \cite{LY}. By modifying their proofs slightly, we can obtain Proposition 3. The following proposition is Theorem B in \cite{Ts}. The proof is easy and essentially follows from the semicontinuity of the entropy.

\begin{prop}
Let $f$ be a contracting Lorenz map with non-flat critical point and $|\frac{{\rm d\theta_{\epsilon}}}{{\rm d Leb}}|_{L^{\infty}} <\frac{d_0}{\epsilon}$, then 
\[ 
h_{\mu_{\infty}}(f) \geq \limsup_{\epsilon \to 0+}h_{\mu_{\epsilon}}(\theta_{\epsilon}).
\]
\end{prop}

\subsection{Proof of Theorem 4}

To prove Theorem 4, we shall prove the following theorem which is a randomized version of \cite{Ts1}[Theorem 3]. Note that the Ruelle inequality asserts that $h_{\mu_{\infty}}(f) \leq \int \chi^+(x) d \mu_{\infty} $.
%check this for piecewise continuous case

\begin{thm}
Under the assumption of Theorem 4, we have
\[ \limsup_{\epsilon \to 0+} \int \chi^+(x; \theta_{\epsilon}) d \mu_{\epsilon}(x) \geq \int \chi^+(x) d \mu_{\infty} (x).
\]
\end{thm}

We shall need several lemmas for preparation. Let $\mathcal P$ denote the set of Borel probability measures on $I$ and let $\mathcal T_{\epsilon} : \mathcal P \to \mathcal P$ be defined as
\[
\mathcal T_{\epsilon} m(A) = \int_{-\epsilon}^{\epsilon} m(f_t^{-1} (A)) d \theta_{\epsilon}(t) \mbox{ for each Borel set } A \subset I.
\]
Note that the stationary measure $\mu_{\epsilon}$ for $\theta_{\epsilon}$ is just the fixed point of $\mathcal T_{\epsilon}$. Since $|\frac{{\rm d\theta_{\epsilon}}}{{\rm d Leb}}|_{L^{\infty}} <\frac{d_0}{\epsilon}$, then
\[
P_{\epsilon} (A | x) = \int_A \theta_{\epsilon} (y - f(x)) dy \leq \frac{d_0}{\epsilon} |A|.
\]
Then for each $m \in \mathcal P$ and each Borel set $A \subset I$, we have
\[
\mathcal T_{\epsilon} m(A) = \int_0^1 P_{\epsilon} (A | x) d m(x) \leq \frac{d_0}{\epsilon} |A|.
\]
In particular, this shows that the stationary measure $\mu_{\epsilon}$ for $\theta_{\epsilon}$ is absolutely continuous with $|\frac{{\rm d\mu_{\epsilon}}}{{\rm d Leb}}|_{L^{\infty}} <\frac{d_0}{\epsilon}$.

\begin{lem}
\[ \lim_{\epsilon \to 0+} \int_{|x -c| < \epsilon^2} \log Df(x) d\mu_{\epsilon} =0.
\]
\end{lem}

\begin{proof}
By the remark before this lemma and Lemma 3.1, we have
\begin{align*}
\bigg| \int_{|x -c| < \epsilon^2} \log Df(x) d\mu_{\epsilon} \bigg| & \leq \int_{|x -c| < \epsilon^2} \big| \log Df(x) \big| \frac{d_0}{\epsilon} dx \\
& \leq -\frac{d_0 C_0}{\epsilon} \int_{|x -c| < \epsilon^2} \log |x - c| dx \\
& = -\frac{2 d_0 C_0}{\epsilon} \int_c^{c + \epsilon^2} \log |x - c| dx \\
& = 2 d_0 C_0 \epsilon (1 - 2 \log \epsilon).
\end{align*}
The last term tends to 0 as $\epsilon \to 0$. This finishes the proof.
\end{proof}

\begin{lem}
For any $K >1$ and any $0 < \xi \leq \frac{1}{2}$, there exists a constant $\delta = \delta(K, \xi) >0$, such that if $x \in[0, 1] \setminus O_f^-(c)$ and $\epsilon >0$ satisfying $\epsilon^2 < x -c =: \eta < \delta$, then for any $\underline t \in [-\epsilon, \epsilon]^{\mathbb N}$ and any positive integer $n \leq -K \log \eta$, we have
\[ 
|f^n_{\underline t}(x) - f^{n-1}(c_1^+)| < \xi |f^{n-1}(c_1^+) -c|.
\]
Similarly, if $\epsilon^2 < c - x  < \delta$, then we have
\[ 
|f^n_{\underline t}(x) - f^{n-1}(c_1^-)| < \xi |f^{n-1}(c_1^-) -c|.
\]
\end{lem}
%gap? need to think

\begin{proof}
Let $[[a, b]]$ denote the closed interval with endpoints $a, b$ without specifying their order. By non-flatness, for any $\lambda >0$ there exists a constant $C(\lambda) >1$ such that if $a, b \in (c, 1]$ or $a, b \in [0, c)$  and $|a-b| < \lambda |b-c|$, then for any $x, y \in [[a, b]]$,
\[  \frac{Df(x)}{Df(y)}  \leq C(\lambda).
\]
Moreover, $C(\lambda) \to 1$ as $\lambda \to 0$.

Fix $\lambda = \lambda(K) >0$ small enough such that $3 K \log C(\lambda) < 1/10$. By Theorem 2, there exists an integer $n_0=n_0(K)$ such that for all $ n \geq n_0$,
\[ \left|\frac{1}{n} \log Df^n(c_1^+) \right| < \log C(\lambda).
\] 
Then there exists a constant $C_1 = C_1(K) >1$ such that for any $n \geq 1$,
\[ \frac{1}{C_1 C(\lambda)^{n}} \leq Df^n(c_1^+) \leq  C_1 C(\lambda)^{n}. \eqno{(4.1)}
\]
By Theorem 1, there exists $\delta_1 =\delta_1(K) \in (0, 1)$ and $n_1 = n_1(K)$ such that if $n \geq n_1$, then 
\[
\frac{1}{n} \sum_{\substack{ 0 \leq i < n \\ f^i(c_1^+) \in \mathcal C_{\delta_1} }} \log d(f^i(c_1^+), c) > -\frac{1}{10K}.
\]
Hence 
\[
\log |f^{n-1}(c_1^+) -c| \geq \sum_{\substack{ 0 \leq i < n \\ f^i(c_1^+) \in \mathcal C_{\delta_1} }} \log d(f^i(c_1^+), c) > -\frac{n}{10K},
\]
which implies $|f^{n-1}(c_1^+) -c| > e^{-\frac{n}{10K}}$ provided that $|f^{n-1}(c_1^+) -c| < \delta_1$. Therefore there exists a constant $C_2=C_2(K) >0$ such that for any $n \geq 1$,
\[
|f^{n-1}(c_1^+) -c| \geq \min\{e^{-\frac{n}{10K}}, \delta_1\} > C_2 e^{-\frac{n}{10K}}. \eqno{(4.2)}
\]
Finally, let $C_3 = C_3(K) \in(0, 1)$ be a constant such that  $C_3 < C(\lambda) -1$, then for any $n \geq 1$,
\[ C(\lambda)^n = [1 + (C(\lambda)-1)]^n > (C(\lambda)-1)n > C_3 n. \eqno{(4.3)}
\]

\begin{claim}
For the fixed $\lambda$ given above, there exists a constant $\delta>0$, such that if $x \in[0, 1] \setminus O_f^-(c)$ and $\epsilon >0$ satisfying $\epsilon^2 < x -c =\eta < \delta$, then for any $\underline t \in [-\epsilon, \epsilon]^{\mathbb N}$ and any positive integer $n \leq -K \log \eta$, we have
\[ |f^n_{\underline t}(x) - f^{n-1}(c_1^+)| < \lambda |f^{n-1}(c_1^+) -c|. \eqno{(4.4)}
\]
\end{claim}

We will prove this claim by induction on $n$. For $n \geq 1$, let 
\[ \tau_n = |f^n_{\underline t}(x) - f^{n-1}(c_1^+)|.
\]
If $\delta$ is small enough, then by Lemma 3.1 we have
\[ \tau_1 = |f_{t_1}(x) -c_1^+ | = |f(x) - c_1^+ +t_1| \leq |x-c|^{\alpha} + \epsilon < \eta^{\alpha} + \sqrt{\eta} < 2 \sqrt{\eta}.
\]
Now assume that $(4.4)$ holds for $1 \leq i \leq n < -K \log \eta$, we will prove that it still holds for $n+1$ provided $\delta$ is sufficiently small.
%then n < -K \log \eta or n \leq -K \log \eta

Note that (4.4) implies that for any $1 \leq i \leq n$ from the induction step, the intervals $[[f^{i}_{\underline t}(x), f^{i-1}(c_1^+) ]]$ do not contain the singular point $c$. By the remark at the beginning of this proof and mean value theorem, 
\begin{align*}
\tau_{i+1} &= |f^{i+1}_{\underline t}(x) - f^{i}(c_1^+)| \leq |f^{i+1}_{\underline t}(x) -f(f^{i}_{\underline t}(x))| + |f(f^{i}_{\underline t}(x)) - f(f^{i-1}(c_1^+))|\\
& \leq \epsilon + Df(\xi_i) \tau_i \leq C(\lambda)  Df(f^{i-1}(c_1^+)) \tau_i + \epsilon,
\end{align*}
where $\xi_i \in [[f^{i}_{\underline t}(x), f^{i-1}(c_1^+) ]]$. For $i \geq 1$, let
\[
\sigma_i = \frac{\tau_i}{C(\lambda)^{i-1} D f^{i-1}(c_1^+)}, \ \epsilon_i =  \frac{\epsilon}{C(\lambda)^{i} D f^{i}(c_1^+)}.
\]
Then for $1 \leq i \leq n$,
\[ \sigma_{i+1} \leq \sigma_i + \epsilon_i.
\]
Since $\sigma_1 = \tau_1$, by (4.1) and (4.3), we have
\[
\sigma_{n+1} \leq \sigma_1 + \sum_{i=1}^n \epsilon_i < \tau_1 + nC_1 \epsilon \leq 2 \sqrt{\eta} + \frac{C_1}{C_3} C(\lambda)^n \sqrt{\eta}.
\]
By (4.1) again,
\[
\tau_{n+1} \leq C_1 C(\lambda)^{2n}\left(2 \sqrt{\eta} + \frac{C_1}{C_3} C(\lambda)^n \sqrt{\eta} \right) \leq C_4 C(\lambda)^{3n} \sqrt{\eta}
\]
for some constant $C_4 = C_4(K) >0$. Since $n < -K \log \eta$, by (4.2) and the fact that $3K \log C(\lambda) < 1/10$, we have
\[
\frac{\tau_{n+1}}{|f^n(c_1^+) -c|} \leq \frac{C_4 C(\lambda)^{3n} \sqrt{\eta}}{ C_2 e^{-\frac{n}{10K}}} \leq \frac{C_4 C(\lambda)^{-3K \log \eta} \sqrt{\eta}}{ C_2 e^{-\frac{1}{10K} \cdot 3K \log \eta}} \leq \frac{C_4}{C_2} \eta^{\frac{7}{10}}.
\]
If $\delta>0$ is small enough such that $\frac{C_4}{C_2} \eta^{\frac{7}{10}} < \lambda$, then 
\[ |f^{n+1}_{\underline t}(x) - f^{n}(c_1^+)| < \lambda |f^n(c_1^+) -c|.
\]
This finishes the induction step and the claim is proved. To prove this lemma, it suffices to assume that $\lambda < \xi$ from the beginning.

\end{proof}

%for each or for Lebesgue a.e. x

\begin{lem}
For each $\gamma>0$, the following holds provided $\epsilon, \delta>0$ are sufficiently small: for each $x \in [0, 1]\setminus O_f^-(c)$ and $\underline t \in [-\epsilon, \epsilon]^{\mathbb N}$, then
\[
\limsup_{n \to \infty} \frac{1}{n} \sum_{\substack{0 \leq i < n \\ \epsilon^2 <|f_{\underline t}^i(x) -c| < \delta }} \log Df(f_{\underline t}^i(x)) \geq - \gamma. \eqno{(4.5)}
\]
\end{lem}

\begin{proof}
Let $K >0$ and $0 < \xi < \frac{1}{2}$ be constants to be determined and assume that $\delta \leq \delta(K, \xi)$ where $ \delta(K, \xi) > 0$ is given by the previous lemma. We define a sequence of integers $n_0 < n_1 < \ldots $ as follows. Let $n_0 = \min\{ k \geq 0: \epsilon^2 <|f_{\underline t}^k(x) -c| < \delta\}$. For $i \geq 0$, define $\eta_i = |f_{\underline t}^{n_i}(x) -c| < \delta$, $n_i' = n_i + [ -K \log \eta_i]$ and define inductively
\[
n_{i+1} = \min \{k > n_i' : \epsilon^2 <|f_{\underline t}^i(x) -c| < \delta \}.
\]
We may assume that $n_i$ are well-defined for all $i \geq 0$, for otherwise (4.5) is obvious. Without loss of generality, we also assume that whenever $\epsilon^2 <|f_{\underline t}^{n_i}(x) -c| < \delta$, $f_{\underline t}^{n_i}(x)$ are alway on the right side of the singular point $c$.

By Lemma 4.2, for each $i \geq 0$ and $n_i +1 \leq j \leq n_i'$, we have 
\[
|f_{\underline t}^j(x) -f^{j-n_i-1}(c_1^+)| < \xi |f^{j-n_i-1}(c_1^+) -c|.
\]
Then there exists a constant $C(\xi) >0$ such that 
\[
\frac{Df(f_{\underline t}^j(x))}{Df(f^{j-n_i-1}(c_1^+))} \leq C(\xi),
\]
where $C(\xi) \to 1$ as $\xi \to 0$. So for each $M \geq 1$, 
\begin{align*}
&\frac{1}{n_M} \sum_{\substack{j=0 \\ \epsilon^2 <|f_{\underline t}^j(x) -c| < \delta }}^{n_M-1} -\log Df(f_{\underline t}^j(x)) = \frac{1}{n_M} \sum_{i=0}^{M-1} \sum_{\substack{ j =n_i \\ \epsilon^2 <|f_{\underline t}^j(x) -c| < \delta }}^{n_{i+1}} -\log Df(f_{\underline t}^j(x))\\
& = \frac{1}{n_M} \sum_{i=0}^{M-1} \sum_{\substack{ j =n_i \\ \epsilon^2 <|f_{\underline t}^j(x) -c| < \delta }}^{n_i'} 
-\log Df(f_{\underline t}^j(x))\\
& \leq  \frac{1}{n_M} \sum_{i=0}^{M-1} \left[  -\log Df(f_{\underline t}^{n_i}(x))  + \sum_{j=n_i+1}^{n_i'} \left(  -\log  Df(f^{j-n_i-1}(c_1^+)) -\log \frac{Df(f_{\underline t}^j(x))}{Df(f^{j-n_i-1}(c_1^+))}   \right)  \right] \\
& \leq \frac{1}{n_M} \sum_{i=0}^{M-1} \left[  -\log Df(f_{\underline t}^{n_i}(x))  -\log  Df^{n_i'-n_i}(c_1^+)
-\sum_{j=n_i+1}^{n_i'}  \log \frac{Df(f_{\underline t}^j(x))}{Df(f^{j-n_i-1}(c_1^+))}    \right].
\end{align*}
By Lemma 3.1, 
\[
-\log Df(f_{\underline t}^{n_i}(x)) \leq -C_0 \log \eta_i \leq \frac{C_0}{K} (n_{i+1} -n_i) ,
\]
where $C_0$ depending only on $f$. Choose $K$ large enough, then
\[
-\log Df(f_{\underline t}^{n_i}(x)) \leq \frac{\gamma}{3} (n_{i+1} -n_i)
\]
provided $\delta$ is small enough. By Theorem 2, $\chi(c_1^+) =0 = \chi(c_1^-)$. Note that if $\delta$ is small enough, then $n_{i+1} -n_i \geq -K \log \eta_i > -K \log \delta$ is very large. Therefore
\[
- \log Df^{n_i'-n_i}(c_1^+)  \leq \frac{\gamma}{3} (n_i'-n_i) < \frac{\gamma}{3}(n_{i+1} -n_i ).
\]
Finally choosing $\xi>0$ small enough such that $\log C(\xi) \leq \gamma/3$. Therefore,
\[
\frac{1}{n_M} \sum_{\substack{j=0 \\ \epsilon^2 <|f_{\underline t}^j(x) -c| < \delta }}^{n_M-1} -\log Df(f_{\underline t}^j(x)) \leq \gamma.
\] 
This finishes the proof.
\end{proof}

\begin{cor}
\[
\lim_{\delta \to 0+ } \limsup_{\epsilon \to 0+} \int_{\epsilon^2 < |x -c| < \delta} \log Df(x) d\mu_{\epsilon} =0.
\]
\end{cor}

\begin{proof}
Recall that $\mu_{\epsilon}$ is ergodic in the sense that $\mu_{\epsilon} \times \theta_{\epsilon}^{\mathbb N}$ is ergodic for the skew product $F(x, \underline t)$. Let $\varphi(x, \underline t) = \log Df(x)$ and $E = \{(x, \underline t) :  \epsilon^2 < |x -c| < \delta\}$. Then $\chi_E \cdot \varphi$ is integrable. By Birkhoff's Ergodic Theorem, for $\mu_{\epsilon} \times \theta_{\epsilon}^{\mathbb N}$-a.e. $(x, \underline t)$,
%Why integrable?
\[
\lim_{n \to \infty} \frac{1}{n} \sum_{i=0}^{n-1} \chi_E \cdot \varphi(F^i(x, \underline t)) = \int \chi_E \cdot \varphi d\mu_{\epsilon} \times \theta_{\epsilon}^{\mathbb N}.
\]
Equivalently,
\[
\lim_{n \to \infty} \frac{1}{n}  \sum_{\substack{0 \leq i < n \\ \epsilon^2 <|f_{\underline t}^i(x) -c| < \delta }} \log Df(f_{\underline t}^i(x)) = \int_{\epsilon^2 < |x -c| < \delta} \log Df(x) d\mu_{\epsilon}.
\]
Then this corollary follows from Lemma 4.3.
\end{proof}

\begin{proof}[{\bf Proof of Theorem 5}]
By Birkhoff's Ergodic Theorem and Lemma 4.1, for any $\delta >0$, we have
\begin{align*}
\limsup_{\epsilon \to 0+} & \int \chi^+(x; \theta_{\epsilon}) d \mu_{\epsilon}(x) \\
&\geq \limsup_{\epsilon \to 0+} \int_{|x-c| < \delta} \log Df(x) d \mu_{\epsilon} +  \lim_{\epsilon \to 0+} \int_{|x-c| \geq \delta} \log Df(x) d \mu_{\epsilon}\\
&= \limsup_{\epsilon \to 0+} \int_{\epsilon^2 < |x-c| < \delta} \log Df(x) d \mu_{\epsilon} + \int_{|x-c| \geq \delta} \log Df(x) d \mu_{\infty}.
\end{align*}
If the conclusion is not true, then there exists a constant $\gamma>0$ such that 
\[
\limsup_{\epsilon \to 0+} \int \chi^+(x; \theta_{\epsilon}) d \mu_{\epsilon}(x) <  \int \log Df(x) d \mu_{\infty} -\gamma.
\]
Therefore,
\[
 \limsup_{\epsilon \to 0+} \int_{\epsilon^2 < |x-c| < \delta} \log Df(x) d \mu_{\epsilon}  < \int_{|x-c| < \delta} \log Df(x) d \mu_{\infty} -\gamma.
\]
By Proposition 2, $\log Df(x)$ is integrable, then $\lim_{\delta \to 0} \int_{|x-c| < \delta} \log Df(x) d \mu_{\infty}  =0$. So when $\delta \to 0$, we get a contradiction to Corollary 4.4.
\end{proof}

\subsection{Proof of Theorem 3}

In this subsection we will prove Theorem 3. 

\begin{lem}
Let $f \in \mathcal L_B$ and $\nu$ be an ergodic $f$-invariant probability measure with $\chi^+(x) := \max \{\overline\chi_f(x), 0 \} =0$ for $\nu$-a.e. $x \in [0, 1]$. Then $\nu$ is the unique physical measure $\mu$ of $f$.
\end{lem}

\begin{proof}

Since $f|\mathcal O_f$ is uniquely ergodic, it suffices to show that ${\rm supp}(\nu) \subset \mathcal O_f$. Let $B_f$ denote the attracting basin of the global attractor $\mathcal O_f$. Then ${\rm Leb}(B_f)= 1$ and  $[0, 1]\setminus (B_f \cup \{ c\})$ is a hyperbolic set with 0 Lebesgue measure by Ma{$\tilde {\rm n}$}\'e's Theorem\footnote{Ma{$\tilde {\rm n}$}\'e's theorem is still valid in the context of contracting Lorenz maps.} \cite{MS}[Chapter III, Theorem 5.1]. Note that ${\rm Leb}( B_f \setminus O_f^-(c))= 1$.

Let $\nu$ be an ergodic $f$-invariant probability measure with $\chi^+(x)=0$ for $\nu$-a.e. $x \in [0, 1]$. Clearly ${\rm supp}(\nu)$ is a non-empty closed invariant set contained in $B_f \cup \{ c\}$. We argue by contradiction, assume that there exists $x_0$ such that $x_0 \in {\rm supp}(\nu) \setminus \mathcal O_f$. Since $\nu$ is ergodic, $f|{\rm supp}(\nu)$ is transitive. Since $\omega(y) = \mathcal O_f$ for any $y \in {\rm supp}(\nu)$, $x_0$ cannot be the image of any $y \in {\rm supp}(\nu)$. Thus $f|{\rm supp}(\nu)$ is not a surjection, a contradiction since transitive maps are always onto. This finishes the proof.

\end{proof}

\begin{proof}[{\bf Proof of Theorem 3}]
Let $f \in \mathcal L_B$. For any $\epsilon>0$ small enough, there exists a unique stationary measure $\mu_{\epsilon}$ which is ergodic. Suppose that $\mu_{\epsilon}$ converges to a measure $\mu_{\infty}$ in the weak star topology. Then $\mu_{\infty}$ is an $f$-invariant probability measure. By Theorem 4, $h_{\mu_{\infty}} (f) = \int \chi^+(x) d \mu_{\infty}(x)$.

Let 
\[
\mu_{\infty} = \int \nu d \lambda(\nu)
\]
be the ergodic decomposition of $\mu_{\infty}$. Then 
\[
h_{\mu_{\infty}}(f) = \int h_{\nu}(f) d \lambda(\nu) = \int \left( \int \chi^+(x) d \nu(x) \right) d \lambda(\nu).
\]
By Ruelle inequality, $h_{\nu}(f) \leq \int \chi^+(x) d \nu(x)$. It follows that for each ergodic component $\nu$ of $\mu_{\infty}$, $h_{\nu}(f) = \int \chi^+(x) d \nu(x)$.

It is proved by Ledrappier \cite{L} that for $C^2$ piecewise interval maps, an ergodic invariant measure of positive entropy is absolutely continuous if and only if it satisfies the Margulis-Pesin formula. Since $f$ has no absolutely continuous invariant measure, $h_{\nu}(f)=0$ for each ergodic component $\nu$ of $\mu_{\infty}$. Since $\chi^+(x) $ is non-negative, it follows that $\chi^+(x)=0$ for $\nu$-a.e. $x$. By Lemma 4.5, $\nu= \mu$, hence $\mu_{\infty} = \mu$. This shows that $f$ is stochastically stable.

\end{proof}

\subsection*{Conflict of interest}

The authors declared no potential conﬂicts of interest with respect to the research.

\subsection*{Data availability statement}

No datasets were generated or analysed during the current study.

\subsection*{Acknowledgements}
The authors would like to thank the anonymous referee for comments that improved the presentation. H. Ji was supported by NSFC Grant No.12301103. Q. Wang was supported by the Natural Science Research Project in Universities of Anhui Province under Grant No.2023AH050105.

%\begin{acknowledgements}
%If you'd like to thank anyone, place your comments here
%and remove the percent signs.
%\end{acknowledgements}

% Authors must disclose all relationships or interests that
% could have direct or potential influence or impart bias on
% the work:
%
% \section*{Conflict of interest}
%
% The authors declare that they have no conflict of interest.

% BibTeX users please use one of
%\bibliographystyle{spbasic}      % basic style, author-year citations
%\bibliographystyle{spmpsci}      % mathematics and physical sciences
%\bibliographystyle{spphys}       % APS-like style for physics
%\bibliography{}   % name your BibTeX data base

\begin{thebibliography}{}

\bibitem{AA}
J. F. Alves, V. Araujo. {\it Random perturbations of nonuniformly expanding maps.} Ast\'erisque. {\bf 286} (2003), 25-62.

\bibitem{ABS}
V. Afra$\rm{\breve{i}}$movi$\rm{\check{c}}$, V. Bykov, L. Shilnikov. {\it The origin and structure of the Lorenz attractor.}(Russian) Dokl. Akad. Nauk SSSR {\bf 234} (1977), no. 2, 336–339.


\bibitem{ACT}
A. Arneodo, P. Coullet, C. Tresser. {\it A possible new mechanism for the onset of turbulence.} Phys. Lett. A. {\bf 81} (4) (1981), 197-201.


\bibitem{BV}
V. Baladi, M. Viana. {\it Strong stochastic stability and rate of mixing for unimodal maps.} Ann. Sci. \'Ecole Norm. Sup. (4). {\bf 29} (1996), 483-517.


\bibitem{B} 
P. Brand{$\tilde {\rm a}$}o. {\it Topological attractors of contracting Lorenz maps.} Ann. Inst. H. Poincar\'e C Anal. Non lin\'eaire. {\bf 35} (5) (2018), 1409–1433.



\bibitem{BY}
M. Benedicks, L. S. Young. {\it Absolutely continuous invariant measures and random pertutbations for certain one-dimensional maps.} Ergod. Th. \& Dynam. Sys. {\bf 12} (1992), 13-37.


\bibitem{CD}
H. Cui, Y. Ding. {\it Invariant measures for interval maps with different one-sided critical orders.} Ergod. Th. \& Dynam. Sys. {\bf 35} (2015), 835-853.

\bibitem{dFG2}
E. de Faria, P. Guarino. {\it Real bounds and Lyapunov exponents.} Dist. Cont. Dyn. Sys. Series A. {\bf 36} (2016), 1957-1982.

\bibitem{G}
D. Gaidashev. {\it Renormalization for Lorenz maps of monotone combinatorial types.} Ergod. Th. \& Dynam. Sys. {\bf 39} (2019), 132-158.

\bibitem{GM}
J. Gambaudo, M. Martens. {\it Algebraic topology for minimal Cantor sets.} Ann. Henri Poincar\'e. {\bf 7} (3) (2006), 423-446.

\bibitem{GW}
J. Guckenheimer, R. Williams. {\it Structural stability of Lorenz attractors.} Publ. Math. IHES. {\bf 50} (1979), 307-320.

\bibitem{HS}
J. H. Hubbard, C. Sparrow. {\it The classification of topologically expansive Lorenz maps.} Commun. Pure Appl. Math. XLIII (1990), 431-443. 

\bibitem{Ji}
H. Ji. {\it Note on Lyapunov Exponent for Critical Circle Maps.} Bull Braz Math Soc, New Series. {\bf 53} (2022), 1305–1315. 

\bibitem{KP}
G. Keller, M. St. Pierre. {\it Topological and measurable dynamics of Lorenz maps.} Ergodic theory, analysis, and efficient simulation of dynamical systems, 333–361, Springer, Berlin, 2001. 

\bibitem{Ki1}
Y, Kifer. {\it Ergodic theory of random transformations.} Birkh${\ddot {\rm a}}$user, Boston, 1986.

\bibitem{L}
F. Ledrappier. {\it Some properties of absolutely continuous invariant measures on an interval.} Ergod. Th. \& Dynam. Sys. {\bf 1} (1981), 77-93.

\bibitem{LY}
F. Ledrappier, L. S. Young. {\it Entropy formula for Random transformations.} Probab. Theory Relat. Fields. {\bf 80} (1988), 217-240.

\bibitem{Lo}
E. N. Lorenz. {\it Deterministic nonperiodic flow.} J. Atmos. Sci. {\bf 20} (2) (1963), 130-141.

\bibitem{LW}
S. Li, Q. Wang. {\it The slow recurrence and stochastic stability of unimodal interval maps with wild attractors.} Nonlinearity. {\bf 26} (2013), 1623-1637.


\bibitem{MM}
M. Martens, W. de Melo. {\it Universal models for Lorenz maps.} Ergod. Th. \& Dynam. Sys. {\bf 21} (3) (2001), 833-860.

\bibitem{MW}
M. Martens, B. Winckler. {\it On the hyperbolicity of Lorenz renormalization.} Comm. Math. Phys. {\bf 325} (1) (2013), 185-257.

\bibitem{MW1}
M. Martens, B. Winckler. {\it Physical measures for infinitely renormalizable Lorenz maps.} Ergod. Th. \& Dynam. Sys. {\bf 38} (2018), 717-738.

\bibitem{MW2}
M. Martens, B. Winckler. {\it Instability of renormalization.} https://arxiv.org/pdf/1609.04473


\bibitem{MS}
W. de Melo, S. van Strien. {\it  One-dimensional dynamics.}
Springer-Verlag, Berlin, 1993.

\bibitem{Me}
R. Metzger. {\it Stochastic stability for contracting Lorenz maps and flows.} Comm. Math. Phys. {\bf 212} (2000), 277-296.


\bibitem{Me1}
R. Metzger. {\it Sinai-Ruelle-Bowen measures for contracting Lorenz maps and flows.} Ann. Inst. H. Poincar\'e C Anal. Non lin\'eaire. {\bf 17} (2000), 247-276.



\bibitem{Pr}
F. Przytycki. {\it Lyapunov characteristic exponents are nonnegative.} Proc. Amer. Math. Soc. {\bf 119} (1993), 309-317.


\bibitem{Ro}
A. Rovella. {\it The dynamics of perturbations of the contracting Lorenz attractor.} Bull. Brazil. Math. Soc. {\bf 24} (1993), 233-259. 

\bibitem{S}
W. Shen. {\it On stochastic stability of non-uniformly expanding interval maps.} Proc. London Math. Soc. {\bf 107} (3) (2013), 1091-1134.

\bibitem{Ts}
M. Tsujii. {\it Small random perturbations of one-dimensional dynamical systems and Margulis-Pesin entropy formula.}  Random Comput. Dynam. {\bf 1}  (1992/93), No.1, 59-89.

\bibitem{Ts1}
M. Tsujii. {\it Positive Lyapunov exponents in families of one dimensional dynamical systems.} Invent. Math. {\bf 111} (1993), 113-137.

\bibitem{W}
B. Winckler. {Renormalization of Lorenz Maps.} Phd Thesis KTH, Stockholm, Sweden, 2011.

\end{thebibliography}

% Non-BibTeX users please use

{\textsc{ School of Mathematics and Statistics,
Zhengzhou University, Zhengzhou,
450001,  CHINA}} (e-mail:jihymath@zzu.edu.cn)

{\textsc{ School of Mathematical Sciences, Anhui University, Hefei, 
230601,  CHINA}} (e-mail:qihan@ahu.edu.cn)

\end{document}